\documentclass[12pt,twoside]{amsart}
\usepackage{amsmath, amsthm, amscd, amsfonts, amssymb, graphicx}
\usepackage{enumerate}
\usepackage[colorlinks=true,
linkcolor=blue,
urlcolor=cyan,
citecolor=red]{hyperref}
\usepackage{mathrsfs}

\newtheorem{theorem}{Theorem}[section]
\newtheorem{lemma}{Lemma}[section]
\newtheorem{remark}{Remark}[section]

\newtheorem{corollary}{Corollary}[section]

\numberwithin{equation}{section}

\newcommand\blfootnote[1]{%
  \begingroup
  \renewcommand\thefootnote{}\footnote{#1}%
  \addtocounter{footnote}{-1}%
  \endgroup
}

\begin{document}
\title{Further subadditive matrix inequalities}
\author{I. H. G\"um\"u\c s, H. R. Moradi and M. Sababheh}
\maketitle
\blfootnote{\subjclass{Math. Sub. Class: Primary 47A63, Secondary 47A30, 39B62.}\\
\hspace{1cm}\keywords{Keywords: subadditive matrix inequalities, concave function.}}

\begin{abstract}
Matrix inequalities that extend certain scalar ones have been at the center of numerous researchers' attention. In this article, we explore the
celebrated subadditive inequality for matrices via concave functions and present a reversed version of this result. Our approach will tackle
concave function properties and some delicate manipulations of matrices and inner products. 
\end{abstract}

%------------------------------------------------------------------------------------%
\pagestyle{myheadings} \markboth{\centerline {}}
{\centerline {}} \bigskip \bigskip 
%------------------------------------------------------------------------------------%
%------------------------------------------------------------------------------------%

\section{Introduction}

In 1999, Ando and Zhan proved that an operator monotone function $%
f:[0,\infty)\to [0,\infty)$ satisfies the subadditive inequality \cite{ando} 
\begin{equation}  \label{ando_zhan_intro}
|||f(A+B)|||\leq |||f(A)+f(B)|||,
\end{equation}
for all $n\times n$ positive semidefinite matrices $A,B$ (written $A,B\geq 0$%
) and any unitarily invariant norm $|||\cdot|||$ on the algebra $\mathcal{M}%
_n$ of all complex $n\times n$ matrices, with identity $I$.

In this context, a function $f:[0,\infty)\to [0,\infty)$ is said to be
operator monotone if it preserves the partial order among Hermitian
matrices. That is, if it satisfies $f(A)\leq f(B)$ whenever $A\leq B$ are
two Hermitian matrices. The partial order ``$\leq"$ among Hermitian matrices
is defined by 
\begin{equation*}
A\leq B ~\Leftrightarrow ~ B-A\geq 0.
\end{equation*}

It is quite interesting that a non-negative function $f$ defined on $%
[0,\infty)$ is operator monotone if and only if it is operator concave, in
the sense that for all $A,B\geq 0,$ 
\begin{align*}
f((1-t)A+tB)\geq (1-t)f(A)+tf(B),\;\forall\; 0\leq t\leq 1.
\end{align*}
Later, in 2007, Bourin and Uchiyama proved \eqref{ando_zhan_intro} for
concave functions; a condition that is much weaker than operator monotony
(or operator concavity), \cite{bourin}.

The motivation behind \eqref{ando_zhan_intro} is that a concave function $%
f:[0,\infty)\to [0,\infty)$  necessarily satisfies 
\begin{align}  \label{concave_sub_intro}
f(a+b)\leq f(a)+f(b),\quad a,b\in [0,\infty).
\end{align}
However, an operator concave version of \eqref{concave_sub_intro} is not
true. That is, an operator concave function $f$ does not necessarily satisfy 
\begin{align}  \label{oper_con_sub_intro}
f(A+B)\leq f(A)+f(B)
\end{align}
for the positive semidefinite matrices $A,B$.\newline
In \cite{mia_sub}, \eqref{oper_con_sub_intro} was discussed in details,
where additional assumptions were assumed to obtain different forms of such
inequalities.

Searching the literature, we find no mention for a reverse of %
\eqref{ando_zhan_intro}. Our second and
main goal of this article is to find a positive term $\Gamma$ such that for a concave function $f:[0,\infty)\to
[0,\infty)$, one has 
\begin{equation*}
|||f(A+B)|||+\Gamma \geq |||f(A)+f(B)|||
\end{equation*}
for all positive semidefinite matrices $A,B$ and any unitarily invariant
norm $|||\cdot|||.$ This will be done in Theorem \ref{thm_main_1} and
Corollary \ref{cor_reversed} below. However, due to the difficulty of the
problem, $\Gamma$ will not have an easy form!

Our approach to prove Theorem \ref{thm_main_1} will be a delicate treatment
of concave functions and inner product properties.

\section{Main Results}

In this section, we present our results, where we begin with the discussion
of concave functions inequalities, then we apply those results to
matrices.

Recall that a concave function is distinguished by the fact that its
position above its secants on the interval of concavity. However, if $%
f:[a,b]\to\mathbb{R}$ is concave, then one can easily see that the function $%
g(t)=f((1-t)a+tb)-((1-t)f(a)+tf(b))$ is concave on $[0,1]$. Consequently,
the graph of $g(t)$ is above its secants on $[0,1/2]$ and $[1/2,1].$ This
observation leads to the well known inequality \cite{mitroi,sab_mia,sab_mjom}
\begin{equation}  \label{eq_conc_ref_1}
\left( 1-t \right)f\left( a \right)+tf\left( b \right)\le f\left( \left( 1-t
\right)a+tb \right)+2r\left( \frac{f\left( a \right)+f\left( b \right)}{2}%
-f\left( \frac{a+b}{2} \right) \right),
\end{equation}
where $0\leq t\leq 1$ and $r=\min\{t,1-t\}.$\newline
Noting negativity of $\frac{f\left( a \right)+f\left( b \right)}{2}-f\left( 
\frac{a+b}{2} \right)$, we see how \eqref{eq_conc_ref_1} refines the
inequality $\left( 1-t \right)f\left( a \right)+tf\left( b \right)\le
f\left( \left( 1-t \right)a+tb \right)$ for concave functions. Manipulating
concave inequalities also lead to a reversed version as follows \cite%
{mitroi,sab_laa} 
\begin{equation}  \label{eq_rev_conc}
\left( 1-t \right)f\left( a \right)+tf\left( b \right)+2R\left( f\left( 
\frac{a+b}{2} \right)-\frac{f\left( a \right)+f\left( b \right)}{2}
\right)\ge f\left( \left( 1-t \right)a+tb \right),
\end{equation}
where $0\leq t\leq 1$ and $R=\max\{t,1-t\}.$

Our first result provides a refinement and a reverse for %
\eqref{concave_sub_intro}. The proof will use both \eqref{eq_conc_ref_1} and %
\eqref{eq_rev_conc}. As far as we know, this approach has never been tickled
in the literature.

\begin{theorem}
\label{1} Let $f:\left[ 0,\infty \right)\to \mathbb{R}$ be a concave
function with $f\left( 0 \right)=0$. Then for any $a,b\ge 0$, 
\begin{equation}  \label{2}
\begin{aligned} & 2\left( 1+\frac{\left| a-b \right|}{a+b} \right)\left(
\frac{f\left( a+b \right)}{2}-f\left( \frac{a+b}{2} \right) \right) \\ & \le
f\left( a+b \right)-\left( f\left( a \right)+f\left( b \right) \right) \\ &
\le 2\left( 1-\frac{\left| a-b \right|}{a+b} \right)\left( \frac{f\left( a+b
\right)}{2}-f\left( \frac{a+b}{2} \right) \right). \end{aligned}
\end{equation}
\end{theorem}

\begin{proof}
For $a,b\geq 0$ and $0\leq t\leq 1,$ \eqref{eq_rev_conc} implies  
\begin{equation*}
\left( 1-t \right)f\left( a \right)+tf\left( b \right)+2R\left( f\left( 
\frac{a+b}{2} \right)-\frac{f\left( a \right)+f\left( b \right)}{2}
\right)\ge f\left( \left( 1-t \right)a+tb \right).
\end{equation*}
where $R=\max \left\{ t,1-t \right\}$.\newline
Replacing $a$ by $0$ and $b$ by $x\geq 0,$ we have 
\begin{equation*}
f\left( tx \right)=f\left( tx+\left( 1-t \right)\cdot0 \right)\le \left( 1-t
\right)f\left( 0 \right)+tf\left( x \right)+2R\left( f\left( \frac{x}{2}
\right)-\frac{f\left( 0 \right)+f\left( x \right)}{2} \right).
\end{equation*}
Since $f\left( 0 \right)=0$, the above inequality implies  
\begin{equation*}
f\left( tx \right)\le tf\left( x \right)+2R\left( f\left( \frac{x}{2}
\right)-\frac{f\left( x \right)}{2} \right),
\end{equation*}
where $R=\max \left\{ t,1-t \right\}$ and $0\le t\le 1$.\newline
Applying this inequality twice implies  
\begin{equation*}
\begin{aligned}
   f\left( a+b \right)&=\frac{a}{a+b}f\left( a+b \right)+\frac{b}{a+b}f\left( a+b \right) \\ 
 & \ge f\left( \frac{a}{a+b}\cdot\left( a+b \right) \right)-2R\left( f\left( \frac{a+b}{2} \right)-\frac{f\left( a+b \right)}{2} \right) \\ 
 &\quad +f\left( \frac{b}{a+b}\cdot\left( a+b \right) \right)-2R\left( f\left( \frac{a+b}{2} \right)-\frac{f\left( a+b \right)}{2} \right) \\ 
 & =f\left( a \right)+f\left( b \right)-4R\left( f\left( \frac{a+b}{2} \right)-\frac{f\left( a+b \right)}{2} \right),  
\end{aligned}
\end{equation*}
where $R=\max \left\{ \frac{a}{a+b},\frac{b}{a+b} \right\}$.\newline
Consequently,  
\begin{equation*}
f\left( a \right)+f\left( b \right)\le f\left( a+b \right)+4R\left( f\left( 
\frac{a+b}{2} \right)-\frac{f\left( a+b \right)}{2} \right).
\end{equation*}
Noting that $R=\max \left\{ \frac{a}{a+b},\frac{b}{a+b} \right\}=\frac{%
a+b+\left| a-b \right|}{2\left( a+b \right)}$, we reach  
\begin{equation*}
f\left( a \right)+f\left( b \right)\le 2\left( 1+\frac{\left| a-b \right|}{%
a+b} \right)\left( f\left( \frac{a+b}{2} \right)-\frac{f\left( a+b \right)}{2%
} \right)+f\left( a+b \right),
\end{equation*}
which proves the first inequality in \eqref{2}.\newline
Now we shall prove the second inequality in \eqref{2}. From %
\eqref{eq_conc_ref_1}, we have 
\begin{equation*}
\left( 1-t \right)f\left( a \right)+tf\left( b \right)\le f\left( \left( 1-t
\right)a+tb \right)+2r\left( \frac{f\left( a \right)+f\left( b \right)}{2}%
-f\left( \frac{a+b}{2} \right) \right),
\end{equation*}
where $r=\min \left\{ t,1-t \right\}$.\newline
This implies, when $a=0,$ 
\begin{equation*}
tf\left( x \right)\le f\left( tx \right)+2r\left( \frac{f\left( x \right)}{2}%
-f\left( \frac{x}{2} \right) \right).
\end{equation*}
Consequently, 
\begin{equation*}
\begin{aligned}
   f\left( a+b \right)&=\frac{a}{a+b}f\left( a+b \right)+\frac{b}{a+b}f\left( a+b \right) \\ 
 & \le f\left( a \right)+2r\left( \frac{f\left( a+b \right)}{2}-f\left( \frac{a+b}{2} \right) \right) \\ 
 & +f\left( b \right)+2r\left( \frac{f\left( a+b \right)}{2}-f\left( \frac{a+b}{2} \right) \right) \\ 
 & =f\left( a \right)+f\left( b \right)+4r\left( \frac{f\left( a+b \right)}{2}-f\left( \frac{a+b}{2} \right) \right),  
\end{aligned}
\end{equation*}
where $r=\min \left\{ \frac{a}{a+b},\frac{b}{a+b} \right\}=\frac{a+b-\left|
a-b \right|}{2\left( a+b \right)}$. This completes the proof of the theorem.
\end{proof}

\begin{remark}
Notice that if $f:\left[ 0,\infty \right)\to \mathbb{R}$ is a concave
function with $f\left( 0 \right)=0$, then for any $a,b\ge 0$, 
\begin{equation*}
0\ge \frac{f\left( a+b \right)}{2}-f\left( \frac{a+b}{2} \right).
\end{equation*}
Since 
\begin{equation*}
\frac{f\left( a+b \right)}{2}\le \frac{f\left( a \right)+f\left( b \right)}{2%
}\le f\left( \frac{a+b}{2} \right),
\end{equation*}
where the first inequality follows from the subadditivity of concave
function and the second inequality follows directly from the definition of a
concave function.
\end{remark}

\begin{corollary}
Let $f:\left[ 0,\infty \right)\to \mathbb{R}$ be a concave function
satisfies $f\left( 0 \right)=0$. Then for any $a,b\ge 0$, 
\begin{equation*}
f\left( \frac{a+b}{2} \right)-\left( \frac{f\left( a \right)+f\left( b
\right)}{2} \right) \le \frac{\left| a-b \right|}{a+b}\left( f\left( \frac{%
a+b}{2} \right)-\frac{f\left( a+b \right)}{2} \right).
\end{equation*}
\end{corollary}

In the sequel, we will present our applications of the above scalar inequalities to matrices. For this purpose, we will need the following well known lemma.
\begin{lemma}
\label{lem_con_inner}(\cite[p. 281]{bhatia}) If $f:J\to\mathbb{R}$ is concave and if 
$A\in\mathbb{M}_n$ is Hermitian with spectrum in $J$, then 
\begin{equation*}
\left<f(A)x,x\right>\leq f\left(\left<Ax,x\right>\right),
\end{equation*}
for all unit vectors $x\in\mathbb{C}^n.$
\end{lemma}
As an application of Theorem \ref{1}, we have the following reversed version
of the celebrated subadditive inequality \eqref{ando_zhan_intro} for concave
functions. For the next two main results, we adopt the notations
$\lambda_{\min}(X)$ and $\lambda_{\max}(X)$ to denote the least and largest eigenvalues of the Hermitian matrix $X\in\mathcal{M}_n$, respectively. 

In the following lemma, we present the reversed version of \eqref{ando_zhan_intro} for the usual operator norm. The unitarily invariant norm version is shown then.

\begin{lemma}
\label{thm_main_1} Let $A,B\in \mathcal{M}_n$ be two positive semidefinite
matrices and let $f:[0,\infty)\to [0,\infty)$ be a concave function, with $%
f(0)=0$. Then 
\begin{equation*}
\begin{aligned}
   \left\| f\left( A \right)+f\left( B \right) \right\|&\le  \frac{\alpha+\beta}{\alpha} \left(2f\left(\frac{\eta}{2}\right)-f(\alpha)\right) +\left\| f\left( A+B \right) \right\|, 
\end{aligned}
\end{equation*}
$\alpha=\lambda_{\min}(A+B), \beta=\lambda_{\max}(|A-B|), \eta=\lambda_{\max}(A+B)$ and  $\|\cdot\|$ is the usual operator norm.
\end{lemma}

\begin{proof}
If $\|x\|=1$, we have $\left\langle \left( A+B \right)x,x \right\rangle \geq \alpha$. Now since $f:[0,\infty)\to [0,\infty)$ is concave with $f(0)=0,$ it follows that $f$ is increasing. Consequently, 
$ -f\left(
\left\langle \left( A+B \right)x,x \right\rangle \right)\le -f(\alpha).
$ This together with the fact that $f$ is increasing imply 
\begin{align}
& 2\left( 1+\frac{\left| \left\langle \left( A-B \right)x,x \right\rangle
\right|}{\left\langle \left( A+B \right)x,x \right\rangle } \right)\left(
f\left( \left\langle \left( \frac{A+B}{2} \right)x,x \right\rangle \right)-%
\frac{f\left( \left\langle \left( A+B \right)x,x \right\rangle \right)}{2}
\right)  \notag \\
& \le \left( 1+\frac{\beta}{\alpha} \right)\left( 2 f\left( \frac{\eta}{2}
\right) -f(\alpha) \right).  \label{needed_1}
\end{align}

Consequently, by applying Theorem \ref{1}, with $a=\left<Ax,x\right>$ and $%
b=\left<Bx,x\right>,$ we have 
\begin{equation*}
\begin{aligned}
  & \left\langle f\left( A \right)+f\left( B \right)x,x \right\rangle  \\ 
 & =\left\langle f\left( A \right)x,x \right\rangle +\left\langle f\left( B \right)x,x \right\rangle  \\ 
 & \le f\left( \left\langle Ax,x \right\rangle  \right)+f\left( \left\langle Bx,x \right\rangle  \right)\hspace{0.2cm}(\text{by Lemma \ref{lem_con_inner}}) \\ 
 & \le 2\left( 1+\frac{\left| \left\langle \left( A-B \right)x,x \right\rangle  \right|}{\left\langle \left( A+B \right)x,x \right\rangle } \right)\left( f\left( \left\langle \left( \frac{A+B}{2} \right)x,x \right\rangle  \right)-\frac{f\left( \left\langle \left( A+B \right)x,x \right\rangle  \right)}{2} \right) \\ 
 &\quad +f\left( \left\langle \left( A+B \right)x,x \right\rangle  \right)\hspace{0.2cm}({\text{by Theorem \ref{1}}}) \\ 
 & \le \left( 1+\frac{\beta}{\alpha} \right)\left( 2 f\left( \frac{\eta}{2}
\right) -f(\alpha) \right) +f\left( \left\langle \left( A+B \right)x,x \right\rangle  \right)\hspace{0.2cm}({\text{by \eqref{needed_1}}}).  
\end{aligned}
\end{equation*}
This implies  
\begin{equation*}
\begin{aligned}
  & \left\langle f\left( A \right)+f\left( B \right)x,x \right\rangle   
 & \le \frac{\alpha+\beta}{\alpha} \left( 2 f\left( \frac{\eta}{2}
\right) -f(\alpha) \right)+f\left( \left\langle \left( A+B \right)x,x \right\rangle  \right), 
\end{aligned}
\end{equation*}
for any unit vector $x\in\mathbb{C}^n$.
Now, by taking supremum over unit vector $x$, and recalling that $f$ is
increasing, we obtain the desired inequality.
\end{proof}

Now we are ready to present the main result in this article, where we show the  unitarily invariant norm version of \eqref{ando_zhan_intro}. In the proof, we will need the following basic lemma \cite[Problem 1.6. 15]{bhatia}.

\begin{lemma}\label{max_orth}
Let $A\in\mathcal{M}_n$ be Hermitian and let $\lambda_1(A)\geq \lambda_2(A)\geq\cdots\geq\lambda_n(A)$ denote all eigenvalues of $A$, counting multiplicities. Then, for $1\leq k\leq n,$
$$
\sum_{i=1}^{k}\lambda_i(A)=\max\sum_{i=1}^{k}\left<Ax_i,x_i\right>,
$$
where the maximum is taken over all sets of $k$ orthogonal vectors $x_1,\cdots,x_k$ in $\mathbb{C}^n$.
\end{lemma}

\begin{theorem}
\label{cor_reversed} Let $A,B\in \mathcal{M}_n$ be two positive matrices
and let $f:[0,\infty)\to [0,\infty)$ be a concave function, with $f(0)=0$.
If  $|||\cdot|||$ is a  unitarily invariant norm on $\mathcal{M}_n$ normalized so that $||| I |||=1$, then
\begin{equation*}
\begin{aligned}
   \left|\left|\left| f\left( A \right)+f\left( B \right) \right|\right|\right|&\le \frac{\alpha+\beta}{\alpha} \left(2f\left(\frac{\eta}{2}\right)-f(\alpha)\right)+\left|\left|\left| f\left( A+B \right) \right|\right|\right|,  
\end{aligned}
\end{equation*}
where $\alpha, \beta$ and $\eta$ are as in Lemma \ref{thm_main_1}.
\end{theorem}

\begin{proof}

Let $x_1,x_2,\cdots,x_n$ be unit eigenvectors corresponding to the
eigenvalues $\lambda_1\geq \lambda_2\geq \cdots\geq \lambda_n$ of $f(A)+f(B)$%
. For simplicity, let
$$\gamma=\frac{\alpha+\beta}{\alpha} \left(2f\left(\frac{\eta}{2}\right)-f(\alpha)\right).$$
Then, for $1\leq k\leq n$, 
\begin{align*}
\sum_{i=1}^{k}\lambda_i(f(A)+f(B))&=\sum_{i=1}^{k}\left<(f(A)+f(B))x_i,x_i%
\right> \\
&\leq \sum_{i=1}^{k} \left<(\gamma I+f(A+B))x_i,x_i\right> \;({\text{by\;Lemma}}\;\ref{thm_main_1})\\
&\leq \sum_{i=1}^{k}\lambda_i(\gamma I+f(A+B))\;({\text{by\;Lemma}}\;\ref{max_orth}).
\end{align*}
Now, since $A$ and $B$ are positive and $f:[0,\infty)\to[0,\infty)$, we have 
\begin{align*}
|||f(A)+f(B)|||_{(k)}\leq ||| \gamma I+ f(A+B) |||_{(k)},
\end{align*}
where $|||\cdot|||_{(k)}$ denotes the ky-Fan norms. From this, it follows
that (see \cite[Theorem IV.2.2, p. 93]{bhatia})  
\begin{align*}
|||f(A)+f(B)|||\leq ||| \gamma I+ f(A+B) |||,
\end{align*}
for any unitarily invariant norm $|||\cdot|||.$ But this latter inequality
implies that  
\begin{align*}
|||f(A)+f(B)|||\leq\gamma||| I ||| +||| f(A+B) |||,
\end{align*}
which completes the proof.
\end{proof}

\vspace{1cm}

{\tiny (I.H. G\"um\"u\c s) Department of Mathematics, Faculty of Arts and Sciences, Ad\i yaman University, Ad\i yaman, Turkey.}

{\tiny \textit{E-mail address:} igumus@adiyaman.edu.tr}

{\tiny \vskip 0.3 true cm }

{\tiny {(H.R. Moradi) Department of Mathematics, Payame Noor University (PNU), P.O. Box 19395-4697, Tehran, Iran.}}

{\tiny \textit{E-mail address:} hrmoradi@mshdiau.ac.ir}

{\tiny \vskip 0.3 true cm }

{\tiny (M. Sababheh) Department of Basic Sciences, Princess Sumaya University for Technology, Al Jubaiha, Amman, Jordan.}

{\tiny \textit{E-mail address:} sababheh@psut.edu.jo; sababheh@yahoo.com}

%-----------------------------------------------------------------------------
%-----------------------------------------------------------------------------

\end{document}